\documentclass{amsart2010}
\usepackage{graphicx}
\usepackage{amsmath}
\usepackage{amsfonts}
\usepackage{amssymb}

\newtheorem{thm}{Theorem}[section]

\newtheorem{lem}[thm]{Lemma}
\newtheorem{prop}[thm]{Proposition}
\theoremstyle{definition}

\theoremstyle{remark}

\numberwithin{equation}{section}
\def\<{\langle}
\def\>{\rangle}
\begin{document}
\title[]{Complex symmetry of Composition operators induced by involutive  Ball automorphisms}
\author{S. Waleed Noor}%
\address{Abdus Salam School of Mathematical Sciences, New Muslim Town, Lahore, Pakistan}
\email{$\mathrm{waleed\_math@hotmail.com}$}

\begin{abstract}Suppose $\mathcal{H}$ is a weighted Hardy space of analytic functions on the unit ball $\mathbb{B}_n\subset\mathbb{C}^n$ such that the composition operator $C_\psi$ defined by $C_{\psi}f=f\circ\psi$ is bounded on $\mathcal{H}$ whenever $\psi$ is a linear fractional self-map of $\mathbb{B}_n$. If $\varphi$ is an involutive Moebius automorphism of $\mathbb{B}_n$, we find a conjugation operator $\mathcal{J}$ on $\mathcal{H}$ such that $C_{\varphi}=\mathcal{J} C^*_{\varphi}\mathcal{J}$. The case $n=1$ answers a question of Garcia and Hammond.
\end{abstract}

\maketitle{}

\section{Introduction}
Let $\mathbb{B}_n$ denote the open unit ball of $\mathbb{C}^n$. A linear fractional self-map $\psi$
of $\mathbb{B}_n$ is a map of the form
\begin{equation}\label{eq:linear fractional map}
\psi(z)=\frac{Az+B}{\langle z,C\rangle+d}
\end{equation}
where $A$ is a linear operator on $\mathbb{C}^n$, with vectors $B,C\in\mathbb{B}_n$ and $d$ a complex number. Fix a vector $a\in\mathbb{B}_n$. Let $P_a$ be the orthogonal projection of $\mathbb{C}^n$ onto the complex line generated by $a$ and let $Q_a=I-P_a$. Setting $s_a=(1-|a|^2)^{1/2}$, we denote by $\varphi_a:\mathbb{B}_n\to\mathbb{B}_n$ the linear fractional map
\begin{equation}\label{eq:Moebius automorphism}
\varphi_a(z)=\frac{a-P_az-s_aQ_az}{1-\langle z,a\rangle}
\end{equation}
which by Section 2.2.1 of Rudin \cite{Rudin} is the involutive Moebius automorphism of $\mathbb{B}_n$ that interchanges $0$ and $a$.

Suppose $\mathcal{H}$ is a weighted Hardy space on $\mathbb{B}_n$ as introduced in \cite{Cowantext}; that is a Hilbert space of analytic functions on $\mathbb{B}_n$ such that the monomials $z^\alpha=z_1^{\alpha_1}\ldots z_n^{\alpha_n}$ for $\alpha=(\alpha_1,\ldots,\alpha_n)\in\mathbb{N}^n$ form an orthogonal basis for $\mathcal{H}$. Additionally, the monomials $z^\alpha$ satisfy $\frac{||z^\alpha||}{||z^\alpha||_{H^2}}=\frac{||z^\gamma||}{||z^\gamma||_{H^2}}$ whenever $|\alpha|=|\gamma|$ for $\gamma\in\mathbb{N}^n$, where $|\gamma|=\gamma_1+\ldots+\gamma_n$ and $||.||$ and $||.||_{H^2}$ are the norms in $\mathcal{H}$ and the classical Hardy space $H^2(\mathbb{B}_n)$ respectively. The sequence $\beta_m=\frac{||z^\alpha||}{||z^\alpha||_{H^2}}$ for $m=|\alpha|$ determines $\mathcal{H}$: the inner product on $\mathcal{H}$ is given by
\begin{equation}\label{eq:inner product}
\langle f,g \rangle=\sum_{m=0}^\infty\langle f_m,g_m\rangle_{H^2}\beta_m^2
\end{equation}
where $f_m,g_m$ are the homogeneous polynomials of degree $m$ in the homogeneous expansions $f=\sum f_m$ and $g=\sum g_m$ on $\mathbb{B}_n$ respectively. If $\psi$ is an analytic self-map of $\mathbb{B}_n$, then the composition operator $C_\psi$ is defined by $C_\psi f=f\circ\psi$ for $f\in\mathcal{H}$. When $\mathcal{H}$ is the classical Hardy space or a weighted Bergman space on $\mathcal{D}=\mathbb{B}_1$, then the boundedness of $C_\psi$ on $\mathcal{H}$ follows from the Littlewood Subordination Theorem. But for $n\geq2$, there exist unbounded composition operators even on the classical spaces \cite{Cowantext}. For our purpose, it is sufficient to assume that $\mathcal{H}$ is a weighted Hardy space such that $C_\psi$ is bounded on $\mathcal{H}$ whenever $\psi$ is a linear fractional self-map of $\mathbb{B}_n$. From here onwards $\mathcal{H}$ shall always denote such a space. Many of the classical Hilbert spaces of analytic functions on $\mathbb{B}_n$ satisfy this requirement: For any real $s>0$, let $H_s$ be the Hilbert space of analytic functions on $\mathbb{B}_n$ with reproducing kernel function
\[
K^s_w(z)=\frac{1}{(1-\langle w,z\rangle)^s}
\]
for $w,z\in\mathbb{B}_n$. These spaces are called generalized weighted Bergman spaces by Zhao and Zhu \cite{Zhu}. Recently Le \cite{Le} showed that $C_\psi$ is bounded on $H_s$ for $s>0$ whenever $\psi$ is a linear fractional self-map of $\mathbb{B}_n$. The space $H_n$ is the classical Hardy space $H^2(\mathbb{B}_n)$ and $H_s$ is the weighted Bergman space $A^2_{s-n-1}(\mathbb{B}_n)$ for $s>n$. When $n\geq 2$, then $H_1$ is the Drury-Arveson space from multi-variable operator theory \cite{Arveson}.

 A bounded operator $T$ on $\mathcal{H}$ is called \emph{complex symmetric} if $T$ has a self-transpose matrix representation with respect to some orthonormal basis of $\mathcal{H}$. This is equivalent to the existence of a \emph{conjugation} (i.e., a conjugate-linear, isometric involution) $C$ such that $T=CT^*C$; and $T$ is called $C$-symmetric. If $z=(z_1,z_2,\ldots,z_n)\in\mathbb{C}^n$ then termwise complex conjugation $\bar{z}=(\overline{z_1},\overline{z_2},\ldots,\overline{z_n})$ is a conjugation operator on $\mathbb{C}^n$. The general study of complex symmetric operators on Hilbert spaces was initiated by Garcia, Putinar and Wogen \cite{Garc2}\cite{Garc3}\cite{Wogen1}\cite{Wogen2}.

 The complex symmetry of weighted composition operators ($M_h C_\psi$; where $M_h$ is multiplication by an analytic function $h$ on $\mathbb{B}_n$) was recently studied by Garcia and Hammond \cite{Garc1} for $n=1$. A complete description of complex symmetric weighted composition operators remains an open problem (even in the case $h\equiv1$). Normal operators are the most obvious examples of complex symmetric operators and the characterization of normal composition operators on $H_s$ for $s>0$ follows from Theorem 8.2 \cite{Cowantext}:

\begin{prop}\label{prop: normal composition}Let $\psi$ be an analytic self-map of $\mathbb{B}_n$. Then $C_\psi$ is normal on $H_s$ for $s>0$ if and only if $\psi(z)=Az$ for some normal linear operator $A$ on $\mathbb{C}^n$ with $||A||\leq1$.
\end{prop}
For a linear operator $V$ on $\mathbb{C}^n$ with $||V||\leq1$, we denote by $C_V$ the composition operator $C_{\phi_V}$ where $\phi_V(z)=Vz$ for $z\in\mathbb{B}_n$. Since $\phi_V$ is a linear fractional map of the form \eqref{eq:linear fractional map}, the operator $C_V$ is bounded on $\mathcal{H}$ and $C_V^*=C_{V^*}$. Define the conjugate-linear operator $J:\mathcal{H}\to\mathcal{H}$ by $(Jf)(z)=\overline{f(\bar{z})}$ for $f\in\mathcal{H}$ and $z\in\mathbb{B}_n$. That $J$ is a conjugation on $\mathcal{H}$ follows from \eqref{eq:inner product} and the fact that $J$ is isometric on $H^2(\mathbb{B}_n)$. The following result is useful for constructing examples of complex symmetric composition operators.

\begin{prop}\label{complex symmetric composition}If $V$ is a complex symmetric linear operator on $\mathbb{C}^n$ with $||V||\leq~1$, then $C_V$ is complex symmetric on $\mathcal{H}$.
\end{prop}
\begin{proof}First suppose that $V$ has a symmetric matrix representation with respect to the standard orthonormal basis of $\mathbb{C}^n$. Then $V$ is complex symmetric with respect to the usual complex conjugation $\bar{z}$ for $z\in\mathbb{C}^n$; that is $\overline{V^*\bar{z}}=Vz$. Hence $C_V$
is $J$-symmetric because for any $f\in\mathcal{H}$ and $z\in\mathbb{C}^n$, we have
\[
(JC_{V^*}Jf)(z)=f(\overline{V^*\bar{z}})=f(Vz)=(C_Vf)(z).
\]
In general, let $U_V$ be a unitary operator on $\mathbb{C}^n$ such that $W=U_V^*VU_V$ has a symmetric matrix with respect to the standard basis of $\mathbb{C}^n$. Then $C_V$ is unitarily equivalent to $C_W$; that is $C_V=C_{U_V^*}C_WC_{U_V}$ where $C_{U_V}$ is unitary on $\mathcal{H}$. Hence if we define the conjugation $\mathcal{J}_V$ on $\mathcal{H}$ by $\mathcal{J}_V=C_{U_V^*}JC_{U_V}$, it follows that $C_V$ is $\mathcal{J}_V$-symmetric.
\end{proof}

It is known that all $2\times2$ complex matrices are complex symmetric \cite{Garc2}. For $n=2$ it follows that $C_V$ is complex symmetric on $\mathcal{H}$ for any linear operator $V$ on $\mathbb{C}^2$. As a consequence of Proposition \ref{prop: normal composition} and Proposition \ref{complex symmetric composition}, if $V$ is complex symmetric on $\mathbb{C}^n$ but not normal, then $C_V$ has the same properties on $\mathcal{H}$.

\section{The Complex symmetry of $C_{\varphi_a}$}
Not every complex symmetric composition operator is of the form $C_V$ for some linear operator $V$ on $\mathbb{C}^n$. The composition operator $C_{\varphi_a}$ is bounded on $\mathcal{H}$ and complex symmetric for all $a\in\mathbb{B}_n$. This is because $C_{\varphi_a}\circ C_{\varphi_a}=I$, and Theorem 2 \cite{Wogen2} states that operators algebraic of degree $2$ are complex symmetric. Our main goal is to construct a conjugation $\mathcal{J}_a$ such that $C_{\varphi_a}$ is $\mathcal{J}_a$-symmetric on $\mathcal{H}$. The particular case $n=1$ resolves a problem of Garcia and Hammond \cite{Garc1}: \emph{If $\varphi$ is an involutive disk automorphism, find an explicit conjugation $\mathcal{J}:\mathcal{H}\to\mathcal{H}$ such that $C_\varphi=\mathcal{J}C_\varphi^*\mathcal{J}$.}

For any $a\in\mathbb{B}_n$, denote by $W_a:\mathcal{H}\to\mathcal{H}$ the unitary part in the polar decomposition of $C_{\varphi_a}$ on $\mathcal{H}$. We need two lemmas for constructing the conjugation $\mathcal{J}_a$ from $W_a$. The first one is a general result about the unitary part in the polar decomposition of a linear involution on a Hilbert space.
\begin{lem}\label{lem:W^2=I}If $T$ is an invertible operator on a Hilbert space with $T^2=I$ and with polar decomposition $T=U|T|$, then $U^2=I$. In particular, $U$ is self-adjoint.
\end{lem}
\begin{proof} We first prove that $|T|^{-1}=|T^*|$.  It is known that if two positive operators $A$ and $B$ commute then $\sqrt{A}$ and $\sqrt{B}$ also commute. Hence
\[
(T^*T)(TT^*)=(TT^*)(T^*T)
\]
implies that $|T||T^*|=|T^*||T|$. So $(|T^*||T|)^2=|T^*|^2|T|^2=(TT^*)(T^*T)=I$. Since the product of two commuting positive operators is again positive, the uniqueness of the square root of a positive operator implies $|T||T^*|=|T^*||T|=I$. We also note that the observation $T^*(T^*T)=(TT^*)T^*$ implies $T^*|T|=|T^*|T^*$. Now since $U=T^*|T|$ it follows that
\[
U^2=(|T^*|T^*)(T^*|T|)=|T^*||T|=I.
\]
Since $U$ is unitary, $U^2=I$ implies $U^*=U$.
\end{proof}

By Lemma \ref{lem:W^2=I}, the unitary operator $W_a$ is self-adjoint and satisfies $W_a^2=I$ on $\mathcal{H}$ for any $a\in\mathbb{B}_n$.

\begin{lem}\label{lem:W_aproperties}If $a\in\mathbb{B}_n\cap\mathbb{R}^n$ then $JW_a=W_aJ$.
\end{lem}
\begin{proof} The polar decomposition $C_{\varphi_a}=W_a|C_{\varphi_a}|$ implies that $W_a=C_{\varphi_a}^*|C_{\varphi_a}|$. We first show that $JC_{\varphi_a}=C_{\varphi_a}J$ for $a\in\mathbb{B}_n\cap\mathbb{R}^n$. For $a=(a_1,\ldots,a_n)\in\mathbb{B}_n\cap\mathbb{R}^n$ and $z=(z_1,\ldots,z_n)\in\mathbb{B}_n$, we have $J\langle z,a\rangle=\overline{\langle \bar{z},a\rangle}={\langle z,a\rangle}$. If $j=1,\ldots,n$, by \eqref{eq:Moebius automorphism} and $P_az=\frac{{\langle z,a\rangle}}{{\langle a,a\rangle}}a$ and $Q_az=z-\frac{{\langle z,a\rangle}}{{\langle a,a\rangle}}a$, the $j$-th component function of $\varphi_a$ denoted by $[\varphi_a]_j$ is
\begin{align*}
[\varphi_a]_j(z)&=\big{[}\frac{a-P_az-s_aQ_az}{1-\langle z,a\rangle}\big{]}_j=\big{[}\frac{(1-\frac{{\langle z,a\rangle}}{{\langle a,a\rangle}}+s_a\frac{{\langle z,a\rangle}}{{\langle a,a\rangle}})a-s_az}{1-\langle z,a\rangle} \big{]}_j\\
&=\frac{(1-\frac{{\langle z,a\rangle}}{{\langle a,a\rangle}}+s_a\frac{{\langle z,a\rangle}}{{\langle a,a\rangle}})a_j-s_az_j}{1-\langle z,a\rangle}.
\end{align*}
By $J\langle z,a\rangle={\langle z,a\rangle}$ it follows that $J[\varphi_a]_j=[\varphi_a]_j$ and hence for any $j=1,\ldots,n$ we have
\[
JC_{\varphi_a}Jz_j=JC_{\varphi_a}z_j=J[\varphi_a]_j(z)=[\varphi_a]_j(z)=C_{\varphi_a}z_j.
\]
Since $C_{\varphi_a}$ and $J$ are both multiplicative operators on $\mathcal{H}$, we have $JC_{\varphi_a}Jz^\alpha=C_{\varphi_a}z^\alpha$ for any multi-index $\alpha\in\mathbb{N}^n$. Therefore $JC_{\varphi_a}J=C_{\varphi_a}$ for $a\in\mathbb{B}_n\cap\mathbb{R}^n$. From this it follows that $JC_{\varphi_a}^*=C_{\varphi_a}^*J$ and hence $J|C_{\varphi_a}|=|C_{\varphi_a}|J$. Since $W_a=C_{\varphi_a}^*|C_{\varphi_a}|$, we get $JW_a=W_aJ$ for $a\in\mathbb{B}_n\cap\mathbb{R}^n$.
\end{proof}

 So for $a\in\mathbb{B}_n\cap\mathbb{R}^n$, lemmas \ref{lem:W^2=I} and \ref{lem:W_aproperties} imply that the conjugate-linear isometry defined by $\mathcal{J}_a=JW_a$ is an involution on $\mathcal{H}$; that is $\mathcal{J}_a^2=(W_aJ)(JW_a)=W^2_a=I$. So $\mathcal{J}_a$ is a conjugation on $\mathcal{H}$ and furthermore:
\begin{thm}\label{cor:S alpha}If $a\in\mathbb{B}_n\cap\mathbb{R}^n$ then $C_{\varphi_a}$ is $\mathcal{J}_a$-symmetric on $\mathcal{H}$.
\end{thm}
\begin{proof}From Lemma \ref{lem:W_aproperties} and the self-adjointness of $W_a$ we get
\begin{align*}
C_{\varphi_a}^*\mathcal{J}_a&=C_{\varphi_a}^*JW_a=C_{\varphi_a}^*W_aJ=C_{\varphi_a}^*W_a^*J=C_{\varphi_a}^*|C_{\varphi_a}|C_{\varphi_a}J\\
&=W_aC_{\varphi_a}J=W_aJC_{\varphi_a}=\mathcal{J}_aC_{\varphi_a}
\end{align*}
and hence that $C_{\varphi_a}$ is $\mathcal{J}_a$-symmetric when $a\in\mathbb{B}_n\cap\mathbb{R}^n$.
\end{proof}

For the general case $a\in\mathbb{B}_n$, we will show that $C_{\varphi_a}$ is unitarily equivalent to $C_{\varphi_{\tilde{a}}}$ for some $\tilde{a}\in\mathbb{B}_n\cap\mathbb{R}^n$. For $\Theta=(\theta_1,\ldots,\theta_n)\in\mathbb{R}^n$, denote by $U_\Theta:\mathcal{H}\to\mathcal{H}$ the bounded composition operator defined by
$(U_\Theta f)(z)=f(e^{i\theta_1}z_1,\ldots,e^{i\theta_n}z_n)$ for $f\in\mathcal{H}$. Then $U_\Theta$ is unitary on $H^2(\mathbb{B}_n)$ and hence also on $\mathcal{H}$ by \eqref{eq:inner product}. If the vector $-\Theta=(-\theta_1,\ldots,-\theta_n)\in\mathbb{R}^n$, then $U^*_\Theta=U_{-\Theta}$. Now fix $a\in\mathbb{B}^n$ and choose $\Theta\in\mathbb{R}^n$ such that $\tilde{a}=(e^{i\theta_1}a_1,\ldots,e^{i\theta_n}a_n)\in\mathbb{R}^n$. Since $\langle U^*_\Theta z,a\rangle=\sum_{j=1}^n e^{-i\theta_j}z_j\bar{a_j}=\langle z,\tilde{a}\rangle$ and $s_a=s_{\tilde{a}}$, for $j=1,\ldots,n$ we have
\begin{align*}
U_\Theta^*C_{\varphi_a}U_\Theta z_j&=U_\Theta^*C_{\varphi_a}e^{i\theta_j}z_j=U_\Theta^*e^{i\theta_j}[\varphi_a]_j=U_\Theta^*e^{i\theta_j}\frac{(1-\frac{{\langle z,a\rangle}}{{\langle a,a\rangle}}+s_a\frac{{\langle z,a\rangle}}{{\langle a,a\rangle}})a_j-s_az_j}{1-\langle z,a\rangle}\\
&=\frac{(1-\frac{{\langle U^*_\Theta z,a\rangle}}{{\langle a,a\rangle}}+s_a\frac{{\langle U^*_\Theta z,a\rangle}}{{\langle a,a\rangle}})e^{i\theta_j}a_j-s_ae^{i\theta_j}U^*_\Theta z_j}{1-\langle U^*_\Theta z,a\rangle}\\
&=\frac{(1-\frac{{\langle z,\tilde{a}\rangle}}{{\langle \tilde{a},\tilde{a}\rangle}}+s_{\tilde{a}}\frac{{\langle z,\tilde{a}\rangle}}{{\langle \tilde{a},\tilde{a}\rangle}})\tilde{a}_j-s_{\tilde{a}}z_j}{1-\langle z,\tilde{a}\rangle}
=C_{\varphi_{\tilde{a}}}z_j
\end{align*}
which implies that $U_\Theta^*C_{\varphi_a}U_\Theta=C_{\varphi_{\tilde{a}}}$ on $\mathcal{H}$. To define the conjugation $\mathcal{J}_a$ when $a\in\mathbb{B}_n$, we let $\mathcal{J}_a=U_\Theta\mathcal{J}_{\tilde{a}}U^*_\Theta$. Hence we have proved
\begin{thm}If $a\in\mathbb{B}_n$ then $C_{\varphi_a}$ is $\mathcal{J}_a$-symmetric on $\mathcal{H}$, where $\mathcal{J}_a=U_\Theta\mathcal{J}_{\tilde{a}}U^*_\Theta$ and $\Theta=(\theta_1,\ldots,\theta_n)\in\mathbb{R}^n$ is chosen such that $\tilde{a}=(e^{i\theta_1}a_1,\ldots,e^{i\theta_n}a_n)\in\mathbb{R}^n$.
\end{thm}
\bibliographystyle{amsplain}

\end{document}